\documentclass[11pt,onesided]{amsart}   	
\usepackage{geometry}                		
\usepackage{graphicx}				
\usepackage{amssymb,mathrsfs}
\usepackage{amssymb,amsthm,amsmath,stmaryrd}
\usepackage{tikz}
\usepackage{tikz-cd}
\usepackage{color}
\usepackage{accents,upgreek,enumerate}
\usepackage[headings]{fullpage}
\usepackage{bm}

\newtheorem{innercustomthm}{Theorem}
\newenvironment{customthm}[1]
  {\renewcommand\theinnercustomthm{#1}\innercustomthm}
  {\endinnercustomthm}
 
\usepackage{color}
\usepackage{mathrsfs}
\usepackage{amssymb}
\usepackage{tikz}
\usepackage{tikz-cd}
\usepackage{bm}
\usepackage{enumerate}
\usetikzlibrary{calc}
\usetikzlibrary{fadings}
\usetikzlibrary{decorations.pathmorphing}
\usetikzlibrary{decorations.pathreplacing}
\usepackage{xcolor}

\DeclareMathAlphabet{\mathpzc}{OT1}{pzc}{m}{it}
\usepackage[backref]{hyperref}
\hypersetup{
  colorlinks   = true,          
  urlcolor     = blue,          
  linkcolor    = blue,          
  citecolor   = purple             
}

\address{Department of Mathematics, Massachusetts Institute of Technology, 77 Massachusetts Avenue, Cambridge, MA 02138}
\email{dhruvr@mit.edu}

\newtheorem{theorem}{Theorem}[subsection]
\newtheorem{corollary}[theorem]{Corollary}

\newtheorem{proposition}[theorem]{Proposition}
\newtheorem{definition}[theorem]{Definition}

\newtheorem{quasi-theorem}[theorem]{Quasi-Theorem}
\newtheorem{blank remark}[theorem]{}

\newtheorem{rem1}[theorem]{Remark}
\newenvironment{remark}{\begin{rem1}\em}{\end{rem1}}

\newtheorem{not1}[theorem]{Notation}

\newcommand{\CC} {{\mathbb C}}          
            
\newcommand{\NN} {{\mathbb N}}		
\newcommand{\PP}{\mathbb{P}}         
		
\newcommand{\RR} {{\mathbb R}}		
\newcommand{\ZZ} {{\mathbb Z}}

\newcommand{\Hom}{\operatorname{Hom}}

\DeclareMathOperator{\val}{val}

\DeclareMathOperator{\spec}{Spec}

\newcommand{\cal}{\mathcal}

\def\cM{{\cal M}}

\newcommand{\plC}{\scalebox{0.8}[1.3]{$\sqsubset$}}

\newcommand{\Mbar}{\overline{\cM}}

\def\trop{\mathrm{trop}}
\def\an{\mathrm{an}}

\makeatletter
\def\blfootnote{\xdef\@thefnmark{}\@footnotetext}
\makeatother

\title[Skeletons of stable maps II]{{\larger \larger S}keletons of stable maps II: {\larger \larger S}uperabundant geometries}

\author[Dhruv Ranganathan]{\vspace{-0.0in}{{\larger D}{\smaller hruv}\ \ {\larger R}{\smaller anganathan}}}
\date{}

\begin{document}

\begin{abstract}
We implement new techniques involving Artin fans to study the realizability of tropical stable maps in superabundant combinatorial types. Our approach is to understand the skeleton of a fundamental object in logarithmic Gromov--Witten theory -- the stack of prestable maps to the Artin fan. This is used to examine the structure of the locus of realizable tropical curves and derive $3$ principal consequences. First, we prove a realizability theorem for limits of families of tropical stable maps. Second, we extend the sufficiency of Speyer's well-spacedness condition to the case of curves with good reduction. Finally, we demonstrate the existence of liftable genus $1$ superabundant tropical curves that violate the well-spacedness condition. 
\end{abstract}

\maketitle

\vspace{-0.2in}

\section{Introduction}

Central to the application of tropical techniques to questions in algebraic geometry are so-called \textit{lifting theorems}. Given a ``synthetic'' tropical object, such as a weighted balanced polyhedral complex, one must understand whether this object is the tropicalization of an algebraic variety. We deal in this paper with the case of curves. The tropical lifting question in this setting asks, \textit{when does an embedded tropical curve in $\RR^n$ arise as the tropicalization of an algebraic curve in a torus over a non-archimedean field?} This question becomes highly nontrivial in the so-called superabundant case and has been the primary obstacle to the application of tropical curve counting techniques in high genus settings. A tropical curve in $\RR^n$ encodes the combinatorial data in a degenerate logarithmic stable map to a toric variety. If the tropical curve is superabundant, i.e. if the tropical deformation space is larger than expected, the obstruction group of this degenerate logarithmic map is nonzero. As a result, such a map may not deform and the tropical curve may fail to be realizable. See Section~\ref{sec: TSM} for a precise definition of superabundance. The earliest realization theorems for superabundant curves are due to Speyer, who observed a subtle combinatorial condition guaranteeing the realizability of superabundant genus $1$ tropical curves~\cite{Sp07}. While there has been substantial additional work in the intervening years, the general question remains mysterious~\cite{BPR16,CFPU,KatLift,Ni09,NS06,Ni15,R15a,Tyo12}. 

In this note, we use recent technical breakthroughs in non-archimedean geometry and the theory of logarithmic maps to provide a conceptual framework in which the realizability question may be approached. To demonstrate the efficacy of this framework, we use it to give simple proofs of three new results for lifting tropical curves. The same framework provides new insight into the structure of realizability conditions more globally -- the locus of realizable tropical curves is given by a union of bend loci of a collection of tropical polynomials in the edge lengths of the tropical curve. 

\subsection{Statement of results} All valued fields appearing in this paper will have equicharacteristic zero. Throughout, the symbol $\plC$ will be used to denote an abstract tropical curve\footnote{We use $\mathscr C$, $C$ and $\plC$ to denote families of curves, single curves, and tropical curves respectively, choosing notation that best approximates the shape of these objects as found in the wild. We thank Dan Abramovich for this most creative of suggestions.}. We say that a parametrized tropical curve $[\plC\to \RR^n]$ is \textit{realizable} if there exists a smooth curve $C$ over a non-archimedean field and a map $[C\to \mathbb G_m^n]$ whose tropicalization is $[\plC\to \RR^n]$. 

\begin{customthm}{A}\label{thm: limiting-realizability}
Let $[\plC_t\to \RR^n]$ for ${t\in [0,1)}$  be a continuously varying family of parametrized tropical curves. Let $[\plC_1\to \RR^n]$ denote the limit of this family in the moduli space of parametrized tropical curves. If $[\plC_t\to \RR^n]$ is realizable for all $t\in [0,1)$, then the limiting map $[\plC_1\to \RR^n]$ is realizable. 
\end{customthm}

Our next result extends the reach of Speyer's well-spacedness condition to the case of elliptic curves with good stable reduction, see Definition~\ref{def: well-spaced}. Let $\plC$ be a tropical curve of genus $1$ with a unique genus $1$ vertex, so in particular the underlying graph of $\plC$ is a tree.  Denote by $\hat \plC$ the tropical curve obtained by first replacing the genus $1$ vertex with a genus $0$ vertex and then adding a self-loop of length $1$ at $v$. 

\begin{customthm}{B}\label{thm: good-reduction}
Let $[\plC\to \RR^n]$ be a parametrized tropical curve of genus $1$ with a unique genus $1$ vertex $v$. Assume that the star of $v$ in $[\hat\plC\to \RR^r]$ is realizable. If $[\plC\to \RR^n]$ is well-spaced, then $[\plC\to \RR^n]$ is realizable. 
\end{customthm}

The work of Speyer shows that the well-spacedness condition is sufficient for realizability of tropical genus $1$ curves. He also proves that this condition is also necessary, with the restriction that the curve is trivalent. The following result shows that outside the trivalent case, well-spacedness can be violated. This complements Baker, Payne, and Rabinoff's generalization of Speyer's condition~\cite[Theorem 6.9]{BPR16}. 

\begin{customthm}{C}\label{thm: non-well-spaced}
Let $n\geq 3$. There exist superabundant parametrized genus $1$ tropical curves $[\plC\to \RR^n]$ that lift to algebraic curves but violate the well-spacedness condition.
\end{customthm}


The new point of view taken in this paper is to attempt to understand the realizability locus inside the moduli space of all parametrized tropical curves as a global tropical geometric object. We do so by studying a fundamental object in logarithmic Gromov--Witten theory -- the space of logarithmic prestable maps to the Artin fan. This is inspired by the insights of Abramovich--Wise, Gross--Siebert, and Ulirsch. By synthesizing these ideas, we are led to the following result. 

Let $X$ be a toric variety with fan $\Delta$ and let $\mathscr
L^\circ_\Gamma(X)$ denote the moduli space of maps from smooth pointed
genus $g$ curves into $X$ with fixed contact orders with the toric
boundary along smooth marked points. In Section~\ref{sec: TSM} a
generalized extended cone complex $T_\Gamma(\Delta)$ is constructed
which parametrizes tropical stable maps with the analogous discrete
data. The generalized cone complex $T^\circ_\Gamma(\Delta)$ is the complement of the
extended faces and parametrizes maps from tropical curves where all
edge lengths are finite.

\begin{customthm}{D}\label{thm: moduli}
There is a continuous tropicalization map 
\[
\trop: \mathscr L_\Gamma^\circ(X)^{\an}\to T^\circ_\Gamma(\Delta),
\]
compatible with evaluation maps and forgetful maps to the moduli space of curves. The locus in $T^\circ_\Gamma(X)$ parametrizing the set of realizable tropical curves is a closed polyhedral set. If $\sigma^\circ\in T^\circ_\Gamma(X)$ is the relative interior of a cell, after passing to a finite cover of $\sigma^\circ$ by the interior of a cone $\tilde\sigma^\circ$, the locus of realizable curves in $\tilde\sigma^\circ$ is the union of bend loci of collections of tropical polynomials. 
\end{customthm}

\subsection{Further discussion} A number of experts have made the informal observation that the condition appearing in Speyer's realizability theorem -- that the minimum of a collection of numbers occurs at least twice -- \textit{resembles the tropical variety of a tropical ideal}. We view Theorem~\ref{thm: moduli} as giving a simple and rigorous explanation for this phenomenon.

While the proof of our lifting theorems only rely on compactness of the realizable locus, the tropical structure is useful for applications to enumerative geometry. This is illustrated, for instance, by the results of Len and the author~\cite{LR15}, in which the polyhedral structure of the realizability locus is used to derive multiplicities for tropical curve counting. 

The above Theorem~\ref{thm: moduli} also contributes to the study of tropical moduli spaces, which have received considerable interest in recent years. These results have aimed at an improved conceptual understanding of information contained in tropical moduli spaces, including applications to enumerative geometry~\cite{CMR14a,CMR14b,R15b} and to the geometry and topology of moduli 
spaces~\cite{CHMR,Cha15,CGP}. The novelty of the present paper is that the tropicalization of $\mathscr L^\circ_\Gamma(X)$ is studied as the tropicalization of a map to a certain toroidal stack -- the stack of prestable maps to the Artin fan. This is in sharp contrast to recent results on tropicalizations for moduli spaces, which have used toroidal structures on the spaces themselves.

The well-spacedness condition has inspired a great deal of research. However, to our knowledge, the results of Baker--Payne--Rabinoff~\cite[Theorem 6.9]{BPR16} are the only known nontrivial necessary conditions for realizability for non-maximally degenerate tropical curves, although it may be possible to extract such conditions using the methods of~\cite{KatLift}. By studying the limits of non-superabundant families of curves and applying Theorem~\ref{thm: limiting-realizability}, one can obtain sufficient conditions for lifting tropical curves in non-maximally degenerate situations, i.e. when vertices carry nonzero genus. Furthermore, Theorem~\ref{thm: good-reduction} exhibits the first instance of a sufficient condition for the realizability for non-maximally degenerate superabundant curves.

In addition to the work of Baker--Payne--Rabinoff mentioned above, Katz extracts a number of necessary conditions for realizability. These stem from interpreting the logarithmic tangent-obstruction complex for maps to toric varieties combinatorially for degenerate maps~\cite{KatLift}. A similar approach is used by Cheung, Fantini, Park, and Ulirsch to prove that in a large range of cases, non-superabundance is a sufficient condition for realizability~\cite{CFPU}. Note however, that limits of non-superabundant tropical curves can often become superabundant. As a result, Theorem~\ref{thm: limiting-realizability} extends the reach of these theorems as well. 

An implicit goal of this paper is to demonstrate the usefulness of the perspective on tropicalization arising from logarithmic prestable maps to Artin fans arising from the work of Abramovich, Ulirsch, and Wise~\cite{AW,U16,U14b} and the insights of Gross and Siebert~\cite{GS13}. Indeed, accepting this technical input, the reader will note that the proofs of our realizability theorems follow from reductions to existing theorems in tropical geometry.  

The first part of this project~\cite{R15b}, which was a chapter in author's doctoral dissertation, studies the tropicalization of the moduli space of logarithmic maps to toric varieties in genus $0$, and it is in this sense that the present paper is a sequel. Superabundance never appears in the genus $0$ setting, and the analogue of Theorem~\ref{thm: moduli} can be used to derive a number of consequences concerning the geometry of the space of maps. We refer to loc. cit. for details. We also note that a similar polyhedrality result as above has been proved by Tony Yu in the context of non-archimedean analytic Gromov--Witten theory using methods that are quite different from ours~\cite{Yu14b,Yu14a}. 

\begin{remark}
In the months between when the first version of this paper appeared on ar$\chi$iv and the final publication, there has been additional progress on tropical realizability. The framework established here has been used by Jensen and the author to prove realization theorems for superabundant tropical curves in the ``chain of cycles'' combinatorial type, in arbitrary genus. This has in turn found applications in Brill--Noether theory, see~\cite[Theorems A \& B]{JR17}. 
\end{remark}

\subsection*{Prerequisites} We assume familiarity with the fundamental concepts of Berkovich geometry and logarithmic structures. A rapid overview of the relevant concepts may be found in the preceding article~\cite[Section 2]{R15b}. We refer the reader to two excellent recent surveys in this area by Abramovich, Chen, and their collaborators~\cite{ACGHOSS,ACMUW}. 

\subsection*{Acknowledgements} During the preparation of this paper, I have benefited greatly from conversations with friends and colleagues including Dori Bejleri, Renzo Cavalieri, Dave Jensen, Sam Payne, Martin Ulirsch, and Ravi Vakil. I'd like to thank David Speyer, in particular, for sharing his insights on realizability in the good reduction case, see Remark~\ref{SpeyerRemark}. I'd also like to thank Dan Abramovich for pointing out a gap in an earlier version of this paper. During important phases of the project I was a student at Brown University and Yale University, and it is a pleasure to acknowledge these institutions here. The final version of this document was greatly improved as a result of the comments of two anonymous referees. Final revisions were completed while the author was a member at the Institute for Advanced Study in Spring 2017.

\noindent
This research was partially supported by NSF grant CAREER DMS-1149054 (PI: Sam Payne).

\section{Tropicalization for maps and their moduli}

\subsection{Logarithmic stable maps} Our central object of study is the moduli space of genus $g$ curves in a projective toric variety $X$ in a fixed curve class, meeting each torus invariant divisor $D_\rho$ at marked points with prescribed contact orders. We work with the compactification of this space of ramified curves, provided by the Abramovich--Chen--Gross--Siebert theory of logarithmic stable maps~\cite{AC11,Che10,GS13}.

Let $X$ be a projective toric variety with dense torus $T$, corresponding to a fan $\Delta$ in the vector space $N_\RR$ spanned by the cocharacters of $T$. Fix integers $g$ and $n$. We study the moduli space $\mathscr L(X)$ of families of minimal stable logarithmic morphisms
\[
f: (C,p_1,\ldots, p_n)\to X,
\]
from nodal pointed genus $g$ curves $(C,p_1,\ldots, p_n)$. 

\begin{remark}
Minimality is a condition on the characteristic monoids of the logarithmic structure on the base of the family of maps. The reader may think about this concept as follows. A logarithmic stable map $[f:C\to X]$ is a stable map in the category of logarithmic schemes. Consequently, given a \textit{logarithmic} scheme $S$, the moduli problem returns the collection of isomorphism classes of logarithmic stable maps over the base logarithmic scheme $S$. However, algebraic stacks are defined over the category of schemes rather than logarithmic schemes. A crucial and beautiful technical achievement of Abramovich, Chen, Gross, and Siebert was to identify an \textit{algebraic} stack $\mathscr L(X)$ that represented a \textit{logarithmic} moduli problem -- the moduli problem for minimal maps. Given a scheme $\underline S$ without a chosen logarithmic structure and a map $\underline S\to \mathscr L(X)$, there is a \textit{minimal} logarithmic structure $\mathscr M_S$ that can be placed on $\underline S$, returning a diagram in the logarithmic category
\[
\begin{tikzcd}
\mathscr C \arrow{d}\arrow{r} & X \\
(\underline S,\mathscr M_S), & \\
\end{tikzcd}
\]
which is a family of logarithmic stable maps. Families obtained by such logarithmic structures are referred to as \textit{minimal} logarithmic stable maps. The {algebraic} stack $\mathscr L(X)$ carries a universal logarithmic structure, and it is by pulling back this structure that we obtain the minimal logarithmic structure $\mathscr M_S$ as above. One can thus interpret this moduli space $\mathscr L(X)$ as a parameter space for minimal logarithmic stable maps, and thus as a stack over the category of schemes. Given a different logarithmic structure $\mathscr M_S'$ on $\underline S$, the universal property of minimality ensures that any map $(\underline S, \mathscr M_S')\to \mathscr L(X)$ factors uniquely through a minimal family. We refer the reader to~\cite[Section 1]{GS13} for an explicit description of these monoids and to~\cite[Section 2]{AC11} for a conceptual discussion.
\end{remark}

Stability for logarithmic maps amounts to stability of the underlying map -- all contracted rational components must have at least $3$ special points and all contracted elliptic components must have at least $1$ special point. We will also fix the contact orders on $C$. The contact order will be recorded by a function
\[
\bm{c}: \{p_1,\ldots,p_n\}\to \Delta(\NN),
\]
where $\Delta(\NN)$ denotes the integral points of the fan $\Delta$. If $\bm{c}(p_i) = e\cdot v_\rho$ where $v_\rho$ is a primitive generator of the ray corresponding to $D_\rho$, then the curve has contact order $e$ along $D_\rho$ at $p_i$. Note that by Fulton--Sturmfels' description of the Chow cohomology of a complete toric variety as Minkowski weights~\cite{FS97}, the data of $\bm c$ determines an operational curve class $f_\star[C]\in A_{1}(X;\ZZ)$. 


We package the discrete data by the symbol $\Gamma = (g,n,\bm c)$. Let $\mathscr L_\Gamma(X)$ denote the moduli space of minimal logarithmic maps carrying discrete data $\Gamma$ and let $\mathscr L^\circ_\Gamma(X)$ be the locus on which the logarithmic structure is trivial. The following result is established by Abramovich--Chen~\cite[Section 5]{AC11} and Gross--Siebert~\cite[Corollary 4.2]{GS13}.

\begin{theorem}
The moduli space $\mathscr L_\Gamma(X)$ of minimal logarithmic stable maps with discrete data $\Gamma$ is a proper logarithmic algebraic stack with projective coarse moduli space. 
\end{theorem}

Note that in this paper, we will often work with logarithmic maps over valuation rings $\spec(R)$ that are not discretely valued. The natural logarithmic structure in this case is coherent, but not fine. In this case, one may approximate $R$ by sub-DVR's and pass to a limit, see~\cite[Appendix A.1]{AW}. Alternatively, in Theorem~\ref{thm: limiting-realizability}, one can pass to a subsequence of $[0,1)$ converging to $1$, parametrizing tropical maps with rational edge lengths. This will eliminate the need to work with maps over non-discrete valuation rings in the sequel. 

\subsection{Tropical stable maps}\label{sec: TSM}
The purpose of this section is to construct a parameter space for tropical stable maps, which will serve as the target of our tropicalization map. 
 
 \begin{definition}A \textbf{pre-stable $n$-pointed tropical curve of genus $g$}, denoted $\plC$, is a finite graph $G({\plC})$ and three additional data:
\begin{enumerate}
\item To each edge of $e\in G(\plC)$, a length $\ell(e)\in \RR_{\geq 0}\sqcup \{\infty\}$ such that if $e$ is a marked leaf-edge, $\ell(e) = \infty$.
\item To every vertex $v\in G(\plC)$ a genus $g(v)\in \ZZ_{\geq 0}$.
\item A labeling of a subset of the $1$-valent genus $0$ vertices of $\plC$ by the set $\{p_1,\ldots, p_n\}$. 
\end{enumerate}
The \textit{genus} of $\plC$ is defined to be 
\[
g(\plC) = h^1(\plC)+\sum_{v\in \plC} g(v).
\]
The tropical curve $\plC$ is made into a topological space as follows. Each edge $e$ of finite length $\ell(e)$ is identified with $[0,\ell(e)]$. Marked leaf-edges are metrized as $[0,\infty]$. Unmarked edges $e$ with $\ell(e) = \infty$ are topologized as as two glued intervals $[0,\infty]\sqcup_\infty[0,\infty]$. We will refer to a curve $\plC$ whose lengths are finite away from the marked leaves as \textbf{smooth}.
\end{definition}

Informally, $\plC$ is often thought of as being a metric space away from the infinite points, and as a space with a ``singular'' metric when including the infinite points. 

\begin{remark}
The terminology of ``smooth'' here is motivated by the following fact. Given a tropical curve $\plC$, one can consider any family $\mathscr C$ of marked prestable curves over a valuation ring $\spec(R)$, such that the skeleton of $\mathscr C$ is $\plC$. The general fiber $\mathscr C_\eta$ of this family is smooth if and only if the non-leaf edge lengths of $\plC$ are all finite, i.e. if $\plC$ is smooth in the above sense. To see this, consider an edge $e$ of $\plC$ corresponding to a node $q$ of $\mathscr C_0$. Formally locally near $q$, the total family may be described as $xy = f$, where $f\in R$, and the valuation of $f$ is identified with the length of $e$. The valuation of $f$ is infinity if and only if $f$ is zero. In turn $f$ is zero if and only if the node $q$ persists in the generic fiber of $\mathscr C$, see~\cite{ACP,BPR16}. 
\end{remark}

Recall that a morphism $\phi: \Sigma_1\to\Sigma_2$ between polyhedral complexes is a map on the underlying point sets such that each polyhedron in $\Sigma_1$ is mapped to a polyhedron in $\Sigma_2$. 

\begin{definition}\label{def: tsm}
A \textbf{tropical stable map from a smooth tropical curve} is a continuous and proper morphism of polyhedral complexes
\[
f:(\plC,\{p_1,\ldots, p_n\}) \to (\overline \Delta,\partial \Delta)
\]
where $(\plC,\{p_1,\ldots, p_n\})$ is a smooth $n$-marked abstract tropical curve, such that the following conditions are satisfied.
\begin{enumerate}[(TSM1)]
\item For each edge $e\in \plC$, the direction of $f(e)$ is an integral vector. Moreover, upon restriction to $e$, has integral slope $w_e$, taken with respect to this integral direction. This integral slope is referred to as the \textbf{expansion factor} of $f$ along $e$. 
\item The map $f$ is balanced in the usual sense, i.e. at all points of $\plC$ the sum of the derivatives of $f$ in each tangent direction is zero.
\item\label{tsm3} The map $f$ is \textbf{stable}. That is, if $p\in \plC$ has valence $2$, then the image of $\mathrm{Star}(v)$ is not contained in the relative interior of a single cone of $\Delta$.
\end{enumerate}
We will usually suppress the markings from the notation, and simply write the map as $[\plC\to \overline \Delta]$.
\end{definition}

The following definition indexes the ``deformation class'' of a tropical stable map and is obtained by dropping the data of the lengths of the edges of $\plC$. 

\begin{definition}
The \textbf{combinatorial type} of a tropical stable map $[\plC\to \overline \Delta]$ is the data
\begin{enumerate}[(CT1)]
\item The finite graph model $G(\plC)$ underlying $\plC$.
\item For each vertex $v\in G(\plC)$, a cone $\sigma_v\in \Delta$ containing the image of $v$.
\item For each edge $e$, the slope $w_e$ of $f$ restricted to $e$ and the primitive vector $u_e$ in the direction of $f(e)$. 
\end{enumerate}
\end{definition}

\begin{definition}
The \textbf{recession type} of a combinatorial type $\Theta$ is obtained from $[\plC\to \overline \Delta]$ by collapsing all bounded edges of $\plC$ to a single vertex. That is, $\plC$ is a single vertex with genus $g$ and marked edges, and the marked edges are each decorated by a contact order. 
\end{definition}

The following proposition seems to be well-known to experts, and a formal proof in the $g = 0$ case may be found in~\cite[Proposition 2.1]{NS06} or~\cite[Proposition 3.2.1]{R15b}. We give an outline of the argument in general. 

\begin{proposition}
Let $\Theta$ be the combinatorial type of a tropical stable map. The set of all tropical curves $[\plC\to \overline \Delta]$ together with an identification with the type $\Theta$ is parametrized by a cone $\sigma_\Theta$. Furthermore, are finitely many combinatorial types with fixed recession type. 
\end{proposition}

\begin{proof}
Fix a combinatorial type $\Theta = [\plC\to \Delta]$ and let $V$ and $E$ be the vertex and bounded edge sets of $\plC$ respectively. To describe a particular tropical map $[f]$ in this combinatorial type is to assign to each $v_i$, a point $f(v_i)$ in the associated cone $\sigma_{v_i}$ and to each edge $e_j$, a real edge length $\ell_e$. In order that these assignments define a continuous and balanced piecewise-linear map to $\Delta$, we simply need to force that for every edge $e$ connecting to $v_{e1}$ and $v_{e2}$, 
\[
f(v_{e1})-f(v_{e2}) = \ell_ew_e.
\]
Here, in keeping with previous notation, $w_e$ is the vector slope prescribed by the combinatorial type. In other words, the set of tropical curves is the subcone of $\prod_{v\in V} \sigma_v \times \prod_{e\in E}\RR_{\geq 0}$ given by
\[
\sigma_\Theta = \left\{((f(v)_v,(\ell_e)_e)\in \prod_{v\in V} \sigma_v \times \prod_{e\in E}\RR_{\geq 0}: \textrm{For all $e = v_{e1}v_{e2} \in E$, } v_{e1}-v_{e2} = \ell_ew_e\right\}.
\]
These relations cut out a subcone and the result follows. The claimed finiteness follows from the boundedness of combinatorial types for logarithmic maps, as proved in~\cite[Section 3.1]{GS13}.
\end{proof}

Following~\cite[Proposition 2.14]{Mi03}, the \textit{overvalence} of a type $\Theta$ is defined as
\[
\mathrm{ov}(\Theta) = \sum_{p:\val(p)\geq 4} \val(p)-3.
\]
The expected dimension of this cone of tropical maps is 
\[
\dim \sigma_\Theta = (\dim(\Delta)-3)(1-b_1(\plC))+n-\mathrm{ov}(\Theta),
\]
where $b_1(\plC)$ is the first Betti number of the underlying graph of $\plC$. This dimension calculation explained, for instance, in~\cite[Section 1]{NS06}. Just as in the algebraic case, it can be deduced from an analysis of the combinatorial tangent-obstruction complex, i.e. the ``abundancy map'' in~\cite[Definition 4.1]{CFPU}. See also~\cite[Sections 2.4-2.6]{Mi03}. The expected dimension above is a lower bound for the dimension of $\sigma_\Theta$, but the actual dimension is often be larger. For instance, if cycles of $\plC$ are mapped into proper affine subspaces of $|\Delta|$ of high codimension, the dimension of the actual deformation space will be larger than the above expected dimension. 

\begin{definition}
A combinatorial type $\Theta$ is said to be \textbf{superabundant} if the dimension of $\sigma_\Theta$ is strictly larger than the expected dimension.
\end{definition}

Following~\cite[Sections 2.3-2.4]{ACP} there is a natural compactification of any cone $\sigma$ with integral structure obtained as follows. Let $S_\sigma$ be the monoid of integral linear functions on $\sigma$ that are non-negative, i.e. the dual monoid of $\sigma$. The cone $\sigma$ can be recovered as the space of monoid homomorphisms $\Hom(S_\sigma,\RR_{\geq0})$. The \textit{canonical compactification} of $\sigma$ is defined as
\[
\overline \sigma:=\Hom(S_\sigma,\RR_{\geq0}\sqcup \{\infty\}).
\]
Fixing a combinatorial type $\Theta$, the faces of the extended cone $\overline \sigma_\Theta$ parametrizes tropical stable maps where $\plC$ is possibly singular and the vertices of $\plC$ map to the extended faces of $\overline \Delta$. We will have no need to work with such maps directly, so we leave the precise formulation to an interested reader.

\begin{definition}
An \textbf{isomorphism} between two tropical stable maps maps $[\plC_1\to \Delta]$ and $[\plC_2\to \Delta]$ is an isometry of graphs, commuting with the vertex weights and with the map:
\[
\begin{tikzcd}
(\plC_1,p_1,\ldots,p_n)\arrow{dr} \arrow{rr} & & (\plC_2,q_1,\ldots,q_n) \\
& \overline \Delta \arrow{ur}. & 
\end{tikzcd}
\]
An isomorphism of a stable map with itself is said to be an \textbf{automorphism}. Similarly, an \textbf{automorphism of the combinatorial type $\Theta$} is an automorphism of the underlying finite graph $G(\plC)$ preserving the edge directions, their expansion factors, vertex weights, and the cones associated to each vertex. 
\end{definition}

The moduli cones for fixed combinatorial types described above form the local models for the space of tropical stable maps. The globalization is achieved by what are now standard techniques, introduced by Abramovich, Caporaso, and Payne~\cite[Section 4]{ACP} for the moduli spaces of curves. In the maps setting, this can be found in~\cite[Section 3]{R15b} in the genus $0$ case. The changes in higher genus are not substantive. Given any combinatorial type $\Theta$, the faces of $\sigma_\Theta$ also parametrize maps from tropical curves. A moduli cone $\sigma_{\Theta'}$ is a face of $\sigma_{\Theta}$ if and only if 
\begin{enumerate}
\item The source type $G(\plC')$ and $\Theta'$ is obtained from the source graph $G(\plC)$ of $\Theta$ by a (possibly trivial) sequence of edge contractions $\alpha: G\to G'$.
\item Given any vertex $v'\in G(\plC')$ and a vertex $v$ such that $\alpha(v) = v'$, then the cone $\sigma_{v'}$ is a face of $\sigma_v$.
\end{enumerate}
The moduli space of tropical stable maps is defined to be
\[
T_\Gamma(\Delta):= \varinjlim_{\Theta: \mathrm{stable}} (\overline\sigma_\Theta,j_\Theta),
\]
where the objects of the colimit are all combinatorial types of a give recession type. The maps $j_\Theta$ range over all identifications of faces, as detailed above, and over automorphisms of combinatorial types. We will refer to the image of an extended cone $\overline \sigma_\Theta$ in the colimit above as a \textit{cell} of the generalized cone complex.

The topological space constructed above has the structure of a
generalized extended cone complex in the sense of~\cite[Section
2]{ACP}. By forming the union of the images of the ordinary (i.e. non-compact) cones $\sigma_\Theta$ in $T_\Gamma(\Delta)$, we obtain the moduli space $T_\Gamma^\circ(\Delta)$, parametrizing those maps with smooth source graph. 

\subsection{Prestable tropical maps} For our later study, it will be convenient to relax the stability condition above. A \textit{prestable} tropical map to $\Delta$ is a map $[\plC\to \Delta]$ as in Definition~\ref{def: tsm}, but possibly violating the stability condition \textit{(TSM3)}. In other words, $[\plC\to \Delta]$ becomes a stable map after dropping finitely many $2$-valent vertices from $\plC$. There are infinitely many such destabilizations of any combinatorial type. However, there are no substantive changes to the construction above -- prestable tropical curves of a fixed combinatorial type are still parametrized by a cone, though of arbitrarily high dimension. Gluing these cones together exactly as above, we obtain a generalized extended cone complex 
\[
T^{\mathrm{pre}}_\Gamma(\Delta) : = \varinjlim_{\Theta: \mathrm{prestable}} (\overline\sigma_\Theta,j_\Theta).
\]
In analogy with toroidal embeddings that are locally of finite type, one might regard this cone complex as being locally of finite-type. 

The crucial technical connection between the logarithmic maps theory and tropical geometry comes from an elegant observation of Gross and Siebert which we now recall. Let $[C\to X]$ be a logarithmic stable map over a logarithmic point $\spec(P\to \CC)$. As explained at the beginning of this section, a map is said to be \textit{minimal} if the logarithmic structure given by $P$ on $\spec(\CC)$ coincides with the logarithmic structure obtained by pulling back the structure on the moduli space $\mathscr L_\Gamma(X)$ via the underlying map
\[
\spec(\CC)\to \mathscr L_\Gamma(X).
\]

The logarithmic stable map $[C\to X]$ has a well-defined \textit{combinatorial type}. The source graph $\Gamma$ is taken to be the dual graph of $C$, the vertices map to the cone $\sigma$ dual to the stratum containing the generic point of the corresponding component, and the expansion factors along edges are uniquely determined by the contact order, see~\cite[Section 1.4]{GS13}. If $\Theta$ is the combinatorial type of $[C\to X]$, let $T_\Theta$ denote the corresponding moduli cone of tropical curves of this combinatorial type and let $T_\Theta(\NN)$ be the integral points of this cone. 

\begin{theorem}[{Gross and Siebert~\cite[Section 1.5]{GS13}}]\label{thm: rem1.4}
Let $[f:C\to X]$ be a logarithmic (pre)-stable map over $\spec(P\to \CC)$ with combinatorial type $\Theta$. The map $[f]$ is minimal if and only if the dual monoid $\Hom(P,\NN)$ is isomorphic to the monoid $T_\Theta(\NN)$.
\end{theorem}

This result follows from the fact that minimality of the map $[f]$ can be characterized by placing constraints on the characteristic monoids of the base, in this case $\spec(P\to\CC)$. As explained in~\cite[Section 1.6 \& Remark 1.21]{GS13}, the map forces certain ``minimal constraints'' that every such base monoid have to satisfy. This implies that there is a natural injective map $\Hom(P,\NN)\to T_\Theta(\NN)$. The universal property of minimality~\cite[Proposition 1.24]{GS13} forces that, if $P$ is minimal, the map is also surjective and thus an isomorphism. 

\noindent
The same result holds when $X$ replaced with its Artin fan $\mathscr A_X = [X/T]$. 

\subsection{Pointwise tropicalization for logarithmic stable maps}\label{sec: pointwise-trop} Our next goal is to construct a map 
\[
\trop: \mathscr L^\circ_\Gamma(X)^{\an} \to T^\circ_\Gamma(\Delta)
\]
from the analytification of the space of maps from smooth curves with prescribed contact orders, to the space of tropical maps.

Any point $x\in \mathscr L^\circ_\Gamma(X)^{\an}$ may be represented by a map
\[
\spec(K)\to \mathscr L^\circ_\Gamma(X),
\]
where $K$ is a valued field extension of $\CC$. After a ramified base change, the existence of the compactification $\mathscr L_\Gamma(X)$ of $\mathscr L^\circ_\Gamma(X)$ guarantees an extension to
\[
\spec(R)\to \mathscr L_\Gamma(X),
\]
where $R$ is the valuation ring of $K$. By pulling back the universal curve and universal map we obtain a diagram
\[
\begin{tikzcd}
(\mathscr C,s_1,\ldots,s_n) \arrow{r}{f} \arrow{d} & X \\
\spec(R), & 
\end{tikzcd}
\]
where the $s_i$ are horizontal sections determining marked points on the geometric fibers. By~\cite[Section 1.4]{BPR13}, this choice of model determines a retraction of the analytic generic fiber onto a tropical curve 
\[
\mathscr C_\eta^{\an}\to \plC,
\]
with a canonical continuous section 
\[
\plC\to \mathscr C_\eta^{\an}. 
\]
Similarly, there is a deformation retraction
\[
X^{\an}\to \overline \Delta,
\]
of the analytic toric variety onto its skeleton by~\cite{Pay09,Thu07}. Composing the section map with the natural map 
\[
\mathscr C_\eta^{\an}\to X^{\an}\to \overline \Delta,
\] 
we obtain a map $[f^\trop:\plC\to \overline \Delta]$. By~\cite[Theorem 3.3.2]{R15b}, this map is seen to be a tropical stable map from a smooth tropical curve.

\subsection{Tropicalization via the Artin fan} Let $X$ continue to denote a projective toric variety. The \textit{Artin fan} of $X$ is defined to be the stack quotient $\mathscr A_X = [X/T]$. Work of Ulirsch~\cite{U16,U14b} asserts that, at the level of underlying topological spaces, the generic fiber of the map $[X\to \mathscr A_X]$ is canonically identified with the extended tropicalization map 
\[
X^{\an}\to \overline \Delta,
\]
constructed in~\cite{Pay09,Thu07}. This is made precise by the following theorem.

Given a scheme or stack $Y$ defined over a valuation ring $R$, we will use $Y^{\an}_\circ$ to denote Raynaud's generic fiber functor applied to the formal completion of $Y$ along the maximal ideal of $R$. See~\cite{U14b,U16,Yu14b} for background on generic fibers of algebraic stacks. Note that if $Y$ is proper, the generic fiber coincides with the Berkovich analytification, and will drop the $\circ$ in the subscript. 

\begin{theorem}[{\cite[Theorem 1.1]{U14b}}] 
There is a canonical identification of extended cone complexes given by $\mu_\Delta: |\mathscr A(\Delta)^{\an}_\circ|\to N(\Delta)$, making the diagram
\[
\begin{tikzcd}
\phantom{1} & \left|{\mathscr A_X}^{\an}_\circ\right| \arrow{dd}{\mu_\Delta} \\
X^{\an} \arrow[swap]{dr}{\trop} \arrow{ru}{\textnormal{Stack Quotient}} & \phantom{1} \\
\phantom{1} & \overline \Delta \\
\end{tikzcd}
\]
commute. 
\end{theorem}

We obtain the following as an immediate corollary. 

\begin{corollary}\label{cor: trop-via-artin}
Let $f: \mathscr C\to X$ be a logarithmic stable map with smooth generic fiber defined over $\spec(R)$, where $R$ is a valuation ring containing $\CC$. Let $\plC$ be the skeleton of $\mathscr C$. Then, the tropicalization of $f$ coincides with the composite map
\[
\plC \to \mathscr C^{\an}_\eta \to X^{\an}\to |{\mathscr A_X}^{\an}_\circ|
\]
\end{corollary}

\subsection{Maps to $\mathscr A_X$} The purpose of this subsection is to establish a ``global'' version of Corollary~\ref{cor: trop-via-artin} above. In~\cite{AW}, Abramovich and Wise introduce a stack of prestable logarithmic morphisms to the Artin fan $\mathscr A_X$ itself. Fixing the discrete data as before, we will use the following result of theirs. 

\begin{theorem}
The stack $\mathscr L^{\mathrm{pre}}_\Gamma(\mathscr A_X)$ is a logarithmically smooth algebraic stack, locally of finite type, and dimension $3g-3+n$. 
\end{theorem}

The reader will note that the logarithmic smoothness of this stack is in sharp contrast to the geometry of $\mathscr L_\Gamma(X)$, which satisfies Murphy's law\footnote{Taking $X = \PP^n$ with its toric boundary and $\Gamma$ to have transverse contact orders, Murphy's Law for $\mathscr L_\Gamma(X)$ follows immediately from~\cite[Theorem M1a]{Vak06}.} in the sense of Vakil~\cite{Vak06}. Indeed, this remarkable smoothness property was applied in~\cite{R15a} to show that every tropical stable map arose as the tropicalization of a family of stable maps to the Artin fan.

In~\cite{Thu07}, Thuillier defines an analytic generic fiber functor $(\cdot)^\beth$ from the category of schemes, locally of finite type, over trivially valued fields to analytic spaces over trivially valued fields. It should be regarded as the trivially valued version of Raynaud's generic fiber in the nontrivially valued case. Given a scheme $X$, points of $X^\beth$ are, by definition, equivalence classes of maps
\[
\spec(R)\to X,
\]
where $R$ is the valuation ring of a valued field extension of the trivially valued base field $k$. While the usual Berkovich analytification $X^{\an}$ reflects the properness and separatedness of $X$, the analytic generic fiber is \textit{always} compact and Hausdorff.  This is an immediate consequence of the valuative criteria for properness and separatedness. In~\cite[Section 5.2]{U16}, this association $X\mapsto X^\beth$ is extended to algebraic stacks that are locally of finite type. We will apply this analytification functor to the stack $\mathscr L_\Gamma^{\mathrm{pre}}(X)$. 

Any point of $\mathscr L_\Gamma^{\mathrm{pre}}(X)^\beth$ gives rise to a family of logarithmic prestable maps
\[
\begin{tikzcd}
(\mathscr C,s_1,\ldots,s_n) \arrow{r} \arrow{d} & \mathscr A_X \\
\spec(R), & 
\end{tikzcd}
\]
over a valuation ring extending $\CC$. Following the same procedure as in the previous section, passing to skeletons yields a map
\[
\trop: \mathscr L_\Gamma^{\mathrm{pre}}(\mathscr A_X)^\beth\to T^{\mathrm{pre}}_\Gamma(\Delta).
\]
On the other hand, the stack $\mathscr L_\Gamma^{\mathrm{pre}}(X)$ is a toroidal (i.e. logarithmically smooth) stack in the lisse-\'etale topology. There exists a continuous deformation retraction of the analytic space onto an extended cone complex: 
\[
\bm{p}_{\mathscr L}: \mathscr L_\Gamma^{\mathrm{pre}}(\mathscr A_X)^\beth\to \overline\Sigma(\mathscr L_\Gamma^{\mathrm{pre}}(\mathscr A_X)).
\]

Note that this deformation retraction has been constructed in various closely related settings in the literature, see~\cite{ACP,Thu07}. We briefly describe the necessary changes in the setting of Artin stacks. Let $\cal X$ be a toroidal Artin stack. Let $U\to \cal X$ be a smooth chart by a toroidal scheme without self-intersection. Consider the self product 
\[
R = U\times_{\cal X} U\rightrightarrows U,
\]
and note that, since $\cal X$ is toroidal, the arrows above are smooth morphisms. We wish to use the toric charts of $U$ to obtain charts on $R$. This is done as follows. Given a smooth morphism $R\to U$, after shrinking, the \'etale local structure theorem for smooth morphisms~\cite[Tag 039P]{stacks-project} yields commutative diagram
\[
\begin{tikzcd}
R \arrow{d} & V_R \arrow{l}\arrow{r} \arrow{d} & U_\sigma \times \mathbb G_m^r \arrow{d} \\
U & V_U \arrow{l} \arrow{r} & U_\sigma.
\end{tikzcd}
\]
In the diagram above, $V_R$ and $V_U$ are opens, and the rightward horizontal arrows are \'etale. Thus, we obtain toroidal charts for $R$. The crucial point is that the torus factors in the charts for $R$ do not change the skeleton, i.e. the skeletons of $U_\sigma^\beth$ and $(U_\sigma\times\mathbb G_m^r)^\beth$ are canonically identified. The extended skeleton $\overline \Sigma(\cal X)$ of $\mathcal X^\beth$ can now be constructed via a colimit of the skeleton for the diagram $(R\rightrightarrows U)^\beth$. This is carried out using only cosmetic modifications to the arguments already present in the literature, see~\cite{ACP,Thu07,U13, U16}. The remaining details are left to an interested reader.

The main result of this section is the relationship between the generalized extended cone complexes $T^{\mathrm{pre}}_\Gamma(\Delta)$ and $\overline \Sigma(\mathscr L_\Gamma^{\mathrm{pre}}(\mathscr A_X))$.

\begin{theorem}\label{trop-for-the-stack}
There is a commutative diagram of continuous morphisms
\[
\begin{tikzcd}
\mathscr L_\Gamma^{\mathrm{pre}}(\mathscr A_X)^\beth\arrow{rr}{\trop}\arrow[swap]{dr}{\bm p_{\mathscr L}} & & T^{\mathrm{pre}}_\Gamma(\Delta)\\
&\overline\Sigma(\mathscr L_\Gamma^{\mathrm{pre}}(\mathscr A_X))  \arrow[swap]{ur}{\trop_\Sigma}, &
\end{tikzcd}
\]
where 
\[
\trop_\Sigma:\overline \Sigma(\mathscr L_\Gamma^{\mathrm{pre}}(X))\to T^{\mathrm{pre}}_\Gamma(X)
\]
is a finite morphism of generalized extended cone complexes and becomes an isomorphism upon restriction to any cell of the source.
\end{theorem}

\begin{proof}
Consider a point $p\in \mathscr L_\Gamma(\mathscr A_X)$ corresponding to a minimal logarithmic map $[C\to \mathscr A_X]$. Suppose $p$ is contained in $W$, the locally closed stratum of the logarithmically smooth stack $\mathscr L_\Gamma(\mathscr A_X)$ parametrizing maps of combinatorial type $\Theta$. As above, there is a smooth open neighborhood $U\to \mathscr L_\Gamma(\mathscr A_X)$ containing $p$ together with an \'etale map
\[
U\to U_\sigma,
\]
where $U_\sigma = \spec(\CC[S_\sigma])$ is an affine toric variety. It follows from logarithmic smoothness of the moduli stack and the definition of the minimal log structure in~\cite[Section 1.4]{GS13} that the monoid $S_\sigma$ defining $U_\sigma$ is coincides with stalk of the minimal characteristic of $[C\to \mathscr A_X]$. Moreover, by Theorem~\ref{thm: rem1.4} there is a natural identification of the cones $\sigma_\Theta$ and $\sigma$. Thus, the skeleton of $U^\beth$ is naturally identified with the extended cone $\overline \sigma$. The local coordinate description of the retraction map~\cite[Section 6]{U13} shows that the retraction
\[
U^\beth\to \overline \sigma_\Theta
\] 
coincides with the pointwise tropicalization map construction in Section~\ref{sec: pointwise-trop}. Given a point of $U^\beth$, we obtain a family of logarithmic maps over a valuation ring $\spec(R)$. We may write $U^\beth$ as a disjoint union of sets, indexed by the locally closed $W$ stratum in $U$ to which such a family specializes. Let $\overline \sigma_\Theta^\circ$ be the closure of the relative interior of $\sigma_\Theta$ in its compactification. The skeleton $\overline \Sigma(\mathscr L(\mathscr A_X))$ decomposes as
\[
\overline \Sigma(\mathscr L^{\mathrm{pre}}_\Gamma(\mathscr A_X)) = \bigsqcup_W \overline \sigma_{\Theta_W}^\circ/\mathrm{Aut}(\Theta_W),
\]
where $\Theta_W$ is the combinatorial type associated to the generic point of the stratum $W$. Similarly, the moduli space of prestable tropical maps decomposes as 
\[
T_\Gamma^{\mathrm{pre}}(\Delta) = \bigsqcup_\Theta \overline \sigma_\Theta^\circ/\mathrm{Aut}(\Theta).
\]
The skeleton $\overline \Sigma(\mathscr L(\mathscr A_X))$ includes naturally into $\mathscr L^{\mathrm{pre}}_\Gamma(\mathscr A_X)^\beth$, so by composing with the pointwise tropicalization map $\trop$, we obtain
\[
\trop_\Sigma: \overline \Sigma(\mathscr L_\Gamma^{\mathrm{pre}}(\mathscr A_X))\to T_\Gamma^{\mathrm{pre}}(\Delta).
\]
By the above discussion, $\trop_\Sigma$ is an isomorphism upon restriction to any fixed cell of $\overline \Sigma(\mathscr L_\Gamma(\mathscr A_X)))$.
It remains to analyze the strata of these two extended cone complexes. For this, we recall that given any logarithmic stable map $[f:C\to \mathscr A_X]$, there is a map in the category of fine but not necessarily saturated logarithmic stacks $f^{\mathrm{us}}:C\to \mathscr A_X$ by the construction in~\cite[Section 3.6]{R15b}. In particular, fix an the underlying map $[\underline C\to \underline {\mathscr A_X}]$. Given two logarithmic enhancements $f_1:C\to \mathscr A_X$ and $f_2:C\to \mathscr A_X$, the maps $f^{\mathrm{us}}_1$ and $f^{\mathrm{us}}_2$ coincide up to saturation. In local charts, saturation is just the normalization of a local (non-normal) toric model. Since the normalization is finite, the saturation is finite. it follows that for each combinatorial type $\Theta$, there are finitely many strata $ W$ having type $\Theta$. If $\sigma_\Theta$ is a cone of $T_\Gamma^{\mathrm{pre}}(\Delta)$, the set of cones in $\trop_\Sigma^{-1}(\sigma_\Theta)$ is identified with the set of minimal logarithmic enhancements of the same underlying map, namely the map parametrized by the generic point of the stratum of type $\Theta$. In particular, this preimage is a finite union of cones, each mapping isomorphically onto $\sigma_\Theta$. The result follows. 
\end{proof}

We derive the first part of Theorem~\ref{thm: moduli} as a corollary.

\begin{corollary}
There is a continuous tropicalization map 
\[
\trop: \mathscr L_\Gamma^\circ(X)^{\an}\to T_\Gamma(\Delta),
\]
compatible with evaluation maps and forgetful maps to the moduli space of curves.
\end{corollary}

\begin{proof}
Consider the sequence of maps
\[
\trop: \mathscr L_\Gamma^\circ(X)^{\an} \to  \mathscr L_\Gamma(X)^\beth \to \mathscr L_\Gamma^{\mathrm{pre}}(\mathscr A_X)^\beth \to T^{\mathrm{pre}}_\Gamma(\Delta).
\]
Since $\mathscr L_\Gamma(X)$ is proper, its formal fiber coincides with its analytification, and first map is simply the analytification of the inclusion of an open subscheme. The second arrow is constructed by applying the functor $(\cdot)^\beth$ for stacks to the natural map from $\mathscr L_\Gamma(X)$ to $\mathscr L_\Gamma(\mathscr A_X)$, see~\cite[Section 5]{U16}. The final map is the pointwise tropicalization constructed in the preceding theorem. The composition is clearly continuous, since each arrow is continuous. Forgetful morphisms to $\Mbar_{g,n}$ and evaluation maps to $X$ are all logarithmic, so functoriality follows from Theorem~\ref{trop-for-the-stack} above and general results on functoriality for tropicalization~\cite[Theorem 1.1]{U13}.
\end{proof}

As a consequence of the theorem, we rephrase the tropical lifting question as follows.

\begin{center}
 \textit{When does a tropical stable map $[\plC\to \Delta]$ lie in the image of the continuous map $\trop$ above?}
\end{center}

The next section takes advantage of the continuity of this map to establish the main applications.

\section{Proofs of lifting theorems}

\subsection{Polyhedrality} Let $p =  [C\to X]$ be a minimal logarithmic stable map and consider the associated map $[C\to \mathscr A_X]$. Let $U$ be a toric neighborhood of $[C\to \mathscr A_X]$ in $\mathscr L_\Gamma^{\mathrm{pre}}(\mathscr A_X)$. Let $Z$ be the local model in the smooth topology of the moduli space $\mathscr L_\Gamma(X)$ near $p$. After possibly shrinking $Z$, the points of the compact analytic space $Z^\beth$ correspond to families over valuation rings of logarithmic stable maps whose special fiber, after composition with $X\to \mathscr A_X$, lies in $U$. 

Let $Q$ be the stalk of the minimal characteristic monoid at $p$. As discussed in the previous section, the toric neighborhood $U$ above admits an \'etale map 
\[
U\to \spec(\CC[\![Q]\!])\times \mathbb G_m^r.
\]
The skeleton of $U^\beth$ is naturally identified with $\Hom(Q,\RR_{\geq 0})$. By the results of the previous section, the locus of realizable tropical curves having the combinatorial type dual to $Q$ is the image of $Z^\beth$ under the natural morphism
\[
Z^\beth \to U^\beth \to \Hom(Q,\RR_{\geq 0}).
\]
The latter map is the retraction map of $U^\beth$ onto its
skeleton. Thus, on the relative interior of each cone, the image of $Z^\beth$ identified with the tropicalization of a subvariety of a toric variety. The fundamental theorem of tropical geometry implies that the image of $Z^\beth$ in the cone above is polyhedral~\cite[Theorem 1.1]{U15}, see also~\cite{BG84,MS14,PP12}. 

Since the realizable locus is polyhedral on the relative interior of each cone, it remains to show that the image of $\mathscr L_\Gamma(X)^\beth$ in $T_\Gamma(\Delta)$ under tropicalization is topologically closed. The analytification of the
moduli space of all minimal logarithmic stable maps $\mathscr
L_\Gamma(X)^{\mathrm{an}}$ is compact, since $\mathscr
L_\Gamma(X)$ is proper. By continuity of the map $\trop$ constructed
in the previous section, its image in
$T_\Gamma^{\mathrm{pre}}(\Delta)$ is compact. Since the tropical
moduli space is Hausdorff, this image is also closed. A tropical curve is realizable if it is the tropicalization of a family of logarithmic maps over a valuation ring whose generic fiber carries the trivial logarithmic structure. In other words, to complete the proof we must identify the image of the locus of maps carrying trivial logarithmic structure, i.e. $\mathscr L_\Gamma^\circ(X)^{\mathrm{an}}$, inside
$T_\Gamma^{\mathrm{pre}}(\Delta)$. Let
$T_\Gamma^{\mathrm{pre},\circ}(\Delta)$ denote the complement of the
extended faces of the generalized extended cone complex
$T_\Gamma^{\mathrm{pre}}(\Delta)$. We claim an equality
of sets
\[
\trop(\mathscr{L}_\Gamma^{\circ}(X)^{\mathrm{an}}) =
\trop(\mathscr{L}_\Gamma(X)^{\mathrm{an}})\cap T_\Gamma^{\mathrm{pre},\circ}(\Delta)
\]
To see this, choose any family of stable maps $[\mathscr C\to
 X]$ over a valuation ring $\spec(R)$ that tropicalizes to a
point of $T_\Gamma^{\mathrm{pre},\circ}(\Delta)$. By definition of the
latter, this means that $\mathscr C$ is a family of nodal curves with
a skeleton such that all edge lengths are finite and thus, the generic
fiber $\mathscr C_\eta$ is smooth. It follows that logarithmic
structure on the base of the generic fiber
of the family $[\mathscr C_\eta\to
 X]$ is trivial, and we obtain the set theoretic equality above. The
 realizable locus is thus closed in
 $T_\Gamma^{\mathrm{pre},\circ}(\Delta)$, which concludes the proof of
Theorem~\ref{thm: moduli}. 
\qed

\begin{remark}
The essential content of the above theorem is that the realizability
conditions are given by the bend loci of the equations that describe
the moduli space of maps locally. By Vakil's Murphy's law, one should
expect these equations to become arbitrarily complicated, and thus one
should also expect arbitrarily high complexity on the piecewise linear side. 
\end{remark}

\subsection{Realizability of limits: Theorem~\ref{thm: limiting-realizability}} Let $[\plC_t\to \Delta]$ for ${t\in [0,1)}$  be a continuously varying family of tropical stable maps. We have an induced moduli map 
\[
[0,1)\to T_\Gamma(\Delta),
\] 
whose image lies in the locus of realizable tropical curves. Since the realizable locus is a closed set, it follows that the limiting map $[\plC_1\to \Delta]$ also lies in this set, and the result follows immediately. 

\qed

\subsection{Well-spacedness for good reduction} We first recall the definition of Speyer's well-spacedness condition. Note that any genus $1$ abstract tropical curve $\plC$ has a unique cycle, possibly a single vertex of genus $1$ or a self-loop at a vertex. A a genus $1$ tropical stable map $[\plC\to \Delta]$ is said to be \textit{superabundant} if the image of the unique cycle $L$ in $|\Delta|$ is contained in a proper affine subspace.  

\begin{definition}\label{def: well-spaced}
Let $[f: \plC\to \Delta]$ be a superabundant genus $1$ tropical stable map. Let $H$ be a hyperplane containing the loop $L$ and consider the subgraph $\plC_H$, the connected component of $f^{-1}(H\cap \plC)$ containing $L$. Denote the $1$-valent vertices of $\plC_H$ by $v_1,\ldots,v_k$ and by $\ell_i$ the distance from $v_i$ to $L$. The map $f$ is \textbf{well-spaced with respect to $H$} if the the minimum of the multiset of distances $\{\ell_1,\ldots, \ell_k\}$ occurs at least twice.
 
\noindent
The map $f$ is said to be \textbf{well-spaced} if it is well-spaced with respect to every hyperplane containing $L$. 
\end{definition}

\subsection{Proof of Theorem~\ref{thm: good-reduction}} Let $[f: \plC\to \Delta]$ be a tropical stable map of genus $1$ such that there is a unique point $p\in \plC$ satisfying $g(p) = 1$. In other words, the underlying graph of $\plC$ is a tree. Let $\hat \plC_t\to \Delta$ be the tropical stable map obtained by replacing replacing the genus function with one such that $g(p) = 0$, and attaching a self-loop at the vertex $p$, where $t\in \RR_+$ is the length of this self-loop. By hypothesis, the star of $p$ in the modified map $[\plC_t\to \Delta]$ is realizable. By applying Speyer's genus $1$ realizability theorem~\cite[Theorem 3.4]{Sp07}, we see that for each value of $t>0$, $\hat \plC_t\to \Delta$ is realizable. Letting $t\to 0$ we obtain a continuous family of realizable tropical curves whose limit is $f$. Since the limiting map must be realizable by Theorem~\ref{thm: limiting-realizability}, the result follows.
\qed

\begin{remark}\label{SpeyerRemark}
David Speyer has communicated to us a different approach to the proof of the above result using techniques akin to those in~\cite{Sp07}. We record the the argument here, in case it may be helpful to the reader. Let $[\plC\to\Delta]$ be a tropical stable map and $v\in \plC$ the genus $1$ vertex, and assume $\mathrm{Star}(v)\to \Delta$ is realizable. Let $a_1, a_2, ..., a_m$ be the directions of the edges pointing outward from $v$ and write 
\[
a_i = (a_{i1}, a_{i2}, \ldots, a_{in}) \in \ZZ^n. 
\]
Let $K$ be a non-archimedean field such that the vertices of $\plC$ map to points of $\Delta$ that are rational over the value group. Let $R$ be the valuation ring and and $k$ the residue field. Since the link is realizable, we may build an elliptic curve $E$ over $k$ with points $z_1, z_2, ... z_m\in E(k)$ such that 
\[
\sum a_i z_i = 0 \in E^n.
\] 

Choose a lifting of $E$ to an elliptic curve $\mathscr E$ over $\spec(R)$ with reduction $E$ over $k$. By the arguments of~\cite[Section 7,8]{Sp07}, to lift the map $[\plC\to \Delta]$ we must lift the points $\{z_i\}$ to the generic fiber of $\mathscr E$ with prescribed relations, preserving $\sum a_i z_i=0$. This can be done by mimicking the calculations in Speyer's original argument, where the lifts of the points $z_i$ lie in the fibers of the reduction map $\mathscr E(R)\to E(k)$. By~\cite[Theorem 6.4]{Silv86}, since we work in residue characteristic $0$, the fibers of this map are isomorphic to the additive group $R^\times$. The calculations are now identical to those in~\cite[Lemma 8.3]{Sp07} and the result follows.
\end{remark}

\begin{remark}
Let $\plC\to \RR^n$ be a genus $1$ tropical stable map as in the statement of Theorem~\ref{thm: good-reduction} and let $p$ be the genus $1$ vertex. We have chosen to formulate the result with the hypothesis that the star of $p$ in the modified curve $\hat\plC\to \RR^n$ is realizable rather than placing a hypothesis on $\plC\to \RR^n$ itself. This choice has been made because the local realizability of $\hat\plC\to \RR^n$ is often easier to check, since one has an explicit coordinate on the nodal $\PP^1$ in the special fiber corresponding to $p$, c.f.~\cite[Proposition 2.8]{CFPU}. One could instead impose that the star of $p$ in $\plC\to \RR^n$ is realizable, and the result continues to hold, as seen in the remark above.
\end{remark}

\subsection{Proof of Theorem~\ref{thm: non-well-spaced}} For the reader's benefit, a ``proof-by-picture'' is given in Figure~\ref{fig: non-well-spaced} below. Let $[\plC\to \Delta]$ be a genus $1$ tropical curve whose cycle is contained in a hyperplane $H$. Assume that this map is well-spaced. Furthermore, assume that the minimum of the distances in Definition~\ref{def: well-spaced} is $0$ and occurs exactly twice. It follows that there exist two vertices $v_1$ and $v_2$, belonging to the cycle $L$, where $\plC$ leaves $H$. Choose a path in $L$ between $v_1$ and $v_2$ and any family of tropical curves $[\plC_t\to \Delta]$ for ${t\in [0,1)}$ such that the distance between $v_1$ and $v_2$ is $(1-t)$. Observe that for each value of $t<1$, the tropical map $[\plC_t\to \Delta]$ is well-spaced, and thus realizable by~\cite[Theorem 3.4]{Sp07}. However, in the limiting map, $\plC_1\to \Delta$, there is a unique point on $L$ at which the cycle $L$ leaves the hyperplane $H$. It follows that the limiting map cannot be well-spaced. However, the map is realizable by Theorem~\ref{thm: limiting-realizability}. 

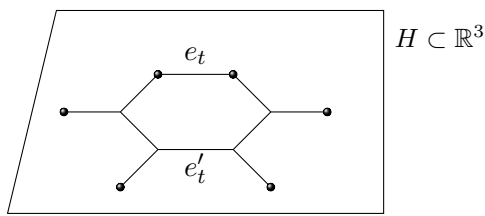
\begin{figure}[h!]
\begin{tikzpicture}
\node at (3.25,1) {\small $H\subset \RR^3$};
\node at (0,0.75) {$e_t$};
\node at (0,-0.75) {$e'_t$};

\draw (-1,0)--(-0.5,0.5)--(0.5,0.5)--(1,0);
\draw (-1,0)--(-0.5,-0.5)--(0.5,-0.5)--(1,0);
\draw (-1,0)--(-1.75,0); 
\draw (1,0)--(1.75,0); 
\draw (-0.5,-0.5)--(-1,-1); 
\draw (0.5,-0.5)--(1,-1); 

\draw (-2.5,-1.35)--(-1.85,1.35)--(2.5,1.35)--(2.5,-1.35)--(-2.5,-1.35);

\draw  [ball color =black] (-0.5,0.5) circle (0.5mm);
\draw  [ball color =black] (0.5,0.5) circle (0.5mm);

\draw  [ball color =black] (-1.75,0) circle (0.5mm);
\draw  [ball color =black] (1.75,0) circle (0.5mm);
\draw  [ball color =black] (-1,-1) circle (0.5mm);
\draw  [ball color =black] (1,-1) circle (0.5mm);
\end{tikzpicture}
\caption{Shown in the figure above is a connected component of the intersection of an embedded tropical curve in $\RR^3$ with a plane $H$. The black vertices indicate the points at which the cycle component leaves $H$. The edges labeled $e_t$ and $e'_t$ have length equal to $(1-t)$.}\label{fig: non-well-spaced}
\end{figure}

\qed

\bibliographystyle{siam} 
\bibliography{Superabundant}

\end{document}